\documentclass{article}
\usepackage{theorem}
\usepackage{eurosym}
\usepackage{amssymb}
\usepackage{amsmath}
\usepackage{amsfonts}
\usepackage[margin=4cm]{geometry}
\usepackage[american]{babel}
\usepackage{verbatim}
\usepackage{graphicx}
\usepackage{color}%
\providecommand{\U}[1]{\protect\rule{.1in}{.1in}}

\newtheorem{theorem}{Theorem}

{\theorembodyfont{\rmfamily}
\newtheorem{example}[theorem]{Example}
}

\newtheorem{lemma}[theorem]{Lemma}

\newtheorem{proposition}[theorem]{Proposition}
{\theorembodyfont{\rmfamily}
\newtheorem{remark}[theorem]{Remark}
}

\newenvironment{proof}[1][Proof]{\noindent\textbf{#1} }{\ \rule{0.5em}{0.5em}}

\newcommand*\re{\mathbb{R}}

\newcommand*\Omegabar{\overline{\Omega}}
\newcommand*\delomega{\partial\Omega}
\newcommand*\intdelomega{\int_{\partial\Omega}}

\newcommand*\dH{d\mathcal{H}^{n-1}}
\newcommand*\HH{\mathcal{H}}
\newcommand*\LL{\mathcal{L}}

\newcommand*\nablas{\nabla_{S^{n-1}}}
\newcommand*\intS{\int_{S^{n-1}}}

\date{}

\begin{document}

\title{On the isoperimetric problem with perimeter density $r^p.$}

\maketitle

\centerline{\scshape Gyula Csat\'{o} }
\medskip
{\footnotesize
\centerline{Facultad de Ciencias F\'isicas y Matem\'aticas, Universidad de Concepci\'on, Concepci\'on, Chile.}
   \centerline{gycsato@udec.cl}

}

\smallskip

\begin{abstract}
In this paper the author studies the isoperimetric problem in $\re^n$ with perimeter density $|x|^p$ and volume density $1.$  
We settle completely the case $n=2,$ completing a previous work by the author: we characterize the case of equality if $0\leq p\leq 1$ and deal with the case $-\infty<p<-1$ (with the additional assumption $0\in\Omega$). In the case $n\geq 3$ we deal mainly with the case $-\infty<p<0,$ showing among others that the results in $2$ dimensions do not generalize for the range $-n+1<p<0.$ 
\end{abstract}

\let\thefootnote\relax
\let\thefootnote\relax

\section{Introduction}

Let $\Omega\subset\re^n$ be a bounded open set with Lipschitz boundary $\partial\Omega.$ We study the inequality
\begin{equation}
 \label{eq:main isop. ineq. intro.}
     \left(\frac{n\LL^n(\Omega)}{\omega_{n-1}}\right)^{\frac{n+p-1}{n}}\leq \frac{1}{\omega_{n-1}}\intdelomega|x|^pd\mathcal{H}^{n-1}(x).
\end{equation}
where $\LL^n(\Omega)$ denotes the $n$ dimensional Lebesgue measure of $\Omega,$  $\mathcal{H}^{n-1}$ is the $n-1$ dimensional Hausdorff measure and $\omega_{n-1}=\mathcal{H}^{n-1}(S^{n-1})$ is the surface area of the unit $n-1$ sphere. For $p=0$ this is the classical isoperimetric inequality.  Note that for balls centered at the origin there is always equality.

For $p<0,$ we introduce a new condition: $\Omega$ shall contain the origin.

The inequality \eqref{eq:main isop. ineq. intro.} is a particular case among a  broader class of problems called isoperimetric problems with densities. Given two positive functions $f,g:\re^n\to \re$ 
one studies the existence of minimizers of
\begin{equation}\label{intro:minimization classical with density}
  I(C)=\inf\left\{\int_{\delomega}g\,d\mathcal{H}^{n-1}:\,\Omega\subset\re^n\text{ and }\int_{\Omega}f\,d\mathcal{L}^n=C\right\}.
\end{equation}
$\intdelomega g\dH$ is called the weighted perimeter.
There are an increasing number of works dealing with different types of weights $f$ and $g,$ see for instance \cite{Rosales Canete Bayle Morgan}, \cite{Cabre 1}, \cite{Cabre 2},  \cite{Chambers Log Convex}, \cite{Figalli-Maggi.LogConvex}, 
\cite{Fusco-Maggi-Pratelli}, \cite{Morgan-Regularity}, \cite{Morgan-Pratelli}. In the case where $f(x)=g(x)=|x|^p$ and $p>0,$  we must mention results appearing in  \cite{Chambers et alt radial p},
\cite{Canete-Miranda-Wittone}, \cite{Carroll-Jacob-Quinn-Walters} and \cite{Dahlberg-Dubbs-Newkirk-Tran}, which, among other results, led to the final surprising fact that the minimizers are hyperspheres passing through the origin. Some other interesting results for the case $f(x)=|x|^q$ and $g(x)=|x|^p$ can be found in \cite{Diaz-Harman-Howe-Thompson}, respectively \cite{Alvino Brock Mercaldo etc.} and \cite{Giosia ...morgans group rk and rm}.

Concerning the inequality \eqref{eq:main isop. ineq. intro.} the following results are already known: In the case $0\leq p<\infty$ the inequality \eqref{eq:main isop. ineq. intro.} always holds.  This has been first proved in \cite{Betta and alt.} for $p\geq 1.$ Another proof of the same result can be found in \cite{Diaz-Harman-Howe-Thompson}, Section 7. Later \eqref{eq:main isop. ineq. intro.} was shown by the author \cite{Csato Diff IntegralEq} for $0\leq p\leq 1$ if $n=2,$ and then by \cite{Alvino Brock Mercaldo etc.} for any $n\geq 3.$ See also the very recent paper \cite{Giosia ...morgans group rk and rm}, which uses variational methods and the study of the Euler equations satisfied by a minimizer. The method of \cite{Alvino Brock Mercaldo etc.} contains a very original interpolation argument with which \eqref{eq:main isop. ineq. intro.} is deduced from the classical isoperimetric inequality. For the sake of completeness we repeat this proof in the special case of starshaped domains, see Proposition \ref{proposition:result of Alvino Brock et alt.}. The general case follows from a more standard, but weighted, symmetrization argument, see \cite{Alvino Brock Mercaldo etc.}.

In the case $p<0$ the only available result is \cite{Csato Diff IntegralEq}. This result shows that if $n=2$ and $-1\leq p\leq 0$ then \eqref{eq:main isop. ineq. intro.} holds true under the additional assumptions $\Omega$ is connected and contains the origin. The motivation for studying negative values of $p$ and adding these additional assumptions come from the singular Moser-Trudinger functional \cite{Adi-Sandeep}. These assumptions arise  naturally in the harmonic transplantation method of Flucher \cite{Flucher} (see \cite{Csato Roy Calc Var} and \cite{Csato Roy Comm PDE}) to establish the existence of extremal functions for the singular Moser-Trudinger functional. Without going into the details, the connection is the following: one uses \eqref{eq:main isop. ineq. intro.} for the level sets of the Greens function $G_{\Omega,0}$ with singularity at $0.$ Obviously these level sets will always contain the origin and will be connected by the maximum principle. 

In the present paper we make the following further contributions. It turns out that there are big differences depending on the dimension $n.$

\smallskip 

\textit{The case $n=1.$} The inequality \eqref{eq:main isop. ineq. intro.} is elementary, but we have included it for completeness. The inequality hols for all $p\in\re\backslash(0,1).$ If $p\in (0,1),$ then minimzers of the weighted perimeter still exist, but they are not intervals centered at the origin.
\smallskip

\textit{The case $n=2.$} 
We settle the case of equality if $0<p<1$ showing that balls centered at the origin are the unique sets satisfying equality. 
Moreover we also deal with $p<-1$ and prove also the sharp form in that case. Thus we settle completely the $2$ dimensional case and the results are 
summarized in Theorem \ref{theorem:main theorem weighted isop. inequlity dim 2}.
\smallskip

\textit{The case $n\geq 3.$}
We completely tackle the case $p<0.$ We prove that \eqref{eq:main isop. ineq. intro.} remains true if $p<-n+1.$ This proof is basically the same  as that of \cite{Betta and alt.}, with a slight difference in the proof of the unicity result. However, the rather surprising result is that the inequality does not hold true for any $p$ between $-n+1$ and $0.$ Particularly interesting is the case $-n+1<p<-n+2.$ In that range the variational method shows that balls centered at the origin are stationary and stable, but nevertheless they are not global minimizers.
We actually prove something stronger: there are no minimizers of the corresponding variational problem \eqref{intro:minimization classical with density} if $-n+1<p<0$ and the infimum is zero.

\section{The 1 dimensional case}\label{section: 1 dim}

In this case $\mathcal{H}^0$ is the counting measure, $\omega_0=2$ and the inequality \eqref{eq:main isop. ineq. intro.} becomes
\begin{equation}
 \label{eq:main isop ineq. n is one}
  \left(\frac{\LL^1(\Omega)}{2}\right)^p\leq \frac{1}{2}\intdelomega|x|^p\,d\mathcal{H}^0(y).
\end{equation}
We assume that $\Omega$ is the union of disjoint open intervals, such that the closed intervals do not intersect. If $p=0,$ and there is just one interval, then the inequality trivially reads as $1=1.$ We have the following proposition.

\begin{proposition}
(i) If $p\geq 1,$ then inequality \eqref{eq:main isop ineq. n is one} holds true for any $\Omega.$ In case of equality $\Omega$ must be single interval and if $p>1$ this interval has to be centered at the origin.
\smallskip

(ii) If $0<p<1,$ then inequality \eqref{eq:main isop ineq. n is one} does not hold true.
\smallskip

(iii) If $p<0,$ then inequality \eqref{eq:main isop ineq. n is one} holds true for all $\Omega$ containing the origin. In case of equality $\Omega$ has to be a single interval centered at the origin.
\end{proposition}

\begin{remark}
In case (ii) one can still ask the question whether there is a minimizer for the weighted perimeter, under the constraint $\LL^1(\Omega)=c.$ One can verify that the unique minimizers are the intervals $(0,c)$ or $(-c,0).$
\end{remark}

\begin{proof}
(i) Let $a=\min\{x|\,x\in \delomega\}$ and $b=\max\{x|\,x\in \delomega\}.$ Then $\LL^1(\Omega)\leq (b-a)\leq |a|+|b|$ and the inequality follows from the fact that the map $s\mapsto s^p$ is increasing and convex.
\smallskip

(ii) Take $\Omega=(0,c),$ for some $c>0.$

\smallskip 
(iii) Since $0\in\Omega$ there exists $a,b\in\delomega$ such that $a<0,b>0$ and $(a,b)\subset\Omega.$ We have $\LL^1(\Omega)\geq (b-a)=|a|+|b|.$ Using that $s\mapsto s^p$ is decreasing and convex one obtains easily the result.
\end{proof}

\section{The 2 dimemsional case}

The following theorem summarizes the works by Betta-Brock-Mercaldo-Posteraro \cite{Betta and alt.}, Csat\'o \cite{Csato Diff IntegralEq} and the results proven in the present paper. In this section $|\Omega|=\LL^2(\Omega)$ shall denote the area of a set and $\sigma$ is the $1$-dimensional Hausdorff measure. We will often just say that a set $\Omega$ is $C^k$ meaning that its boundary $\delomega$ is a $C^k$ curve.

\begin{theorem}
\label{theorem:main theorem weighted isop. inequlity dim 2}
Let $\Omega\subset\re^2$ be a bounded open  Lipschitz set. Regarding the inequality
\begin{equation}
 \label{theorem:eq:main isop. ineq.}
  \left(\frac{|\Omega|}{\pi}\right)^{\frac{p+1}{2}}\leq\frac{1}{2\pi}\int_{\delomega}|x|^pd\sigma,
\end{equation}
the following statements hold true:
\smallskip


(i) If $p\geq 0,$ then \eqref{theorem:eq:main isop. ineq.} holds for all $\Omega.$
\smallskip

(ii) If $-1<p<0,$ then \eqref{theorem:eq:main isop. ineq.} holds for all $\Omega$ connected and containing the origin.
\smallskip 

(iii) If $p\leq -1$ then \eqref{theorem:eq:main isop. ineq.} holds for all $\Omega$ containing the origin.
\smallskip

(iv) In case (i), if $\Omega$ is $C^2$ and there is equality in \eqref{theorem:eq:main isop. ineq.}, then $\Omega$ must be a ball; centered at the origin if $p\neq 0.$ If there is equality in case (ii) and (iii) then also  $\Omega$ must be a ball.
\end{theorem}

\begin{remark}
It follows from Morgan \cite{Morgan-Regularity} that if one has equality in \eqref{theorem:eq:main isop. ineq.} (i.e. $\Omega$ is a minimizer), then $\delomega\backslash\{0\}$ is smooth. On the other hand if one shows the inequality \eqref{theorem:eq:main isop. ineq.} for smooth sets, then it also holds for Lipschitz sets by approximation. So we can work with smooth sets and have to be careful only if $0\in \delomega.$
\end{remark}

\begin{proof}
(i) and (ii) have been proven in \cite{Csato Diff IntegralEq} and \cite{Betta and alt.}, (iii) will be proven in Theorem \ref{theorem:p leq -n+1}. The case of equality has been proven for (ii) in \cite{Csato Diff IntegralEq} and for (iii) it is again a special case of Theorem \ref{theorem:p leq -n+1}. So it remains to deal with the case of equality in (i). This has also been dealt with in \cite{Csato Diff IntegralEq} as long as $0\notin \delomega$ or $p>1.$  The case $0\in \delomega$ is Proposition \ref{proposition:p between 0 and 1 sharp form}
\end{proof}

\smallskip

The rest of this section is devoted to the proof of the following Proposition \ref{proposition:p between 0 and 1 sharp form}. It is based on variational methods and a careful analysis of the resulting Euler-Lagrange equation.


\begin{proposition}
\label{proposition:p between 0 and 1 sharp form}
Suppose $0<p\leq 1$ and $\Omega\subset\re^2$ is a bounded open set with $C^2$ boundary $\delomega.$ If
$$
  \left(\frac{|\Omega|}{\pi}\right)^{\frac{p+1}{2}}=\frac{1}{2\pi}\int_{\delomega}|x|^pd\sigma(x),
$$
then $0$ cannot lie on the boundary $\delomega.$
\end{proposition}

The proof of Proposition \ref{proposition:p between 0 and 1 sharp form} is based on two lemmas. The first one, Lemma \ref{lemma:variational equation p-2 and p}, is the variational formula that we need, establishing the Euler-Lagrange equation satisfied by a minimizer. It is the generalization of the statement that minimizers of the classical isoperimetric problem have constant curvature, with the difference that one introduces a generalized curvature. Such formulae are broadly used to deal with isoperimetric problems with densities, see for instance \cite{Rosales Canete Bayle Morgan}, \cite{Chambers Log Convex}, \cite{Dahlberg-Dubbs-Newkirk-Tran} \cite{Diaz-Harman-Howe-Thompson}, respectively \cite{Morgan Vari Formulae} for a summary. Since in the present case the derivation is very short and elementary, we provide the proof to make the presentation self-contained. Afterwards, in Lemma \ref{lemma:symmetry with respect to axis}, one uses the symmetry properties of the Euler-Lagrange equation to show that 
minimizers are symmetric with respect to any line through the origin and a point $P$ on $\delomega$ with maximal distance from the origin. Such symmetrization arguments are also well known, see for instance Lemma 2.1 in \cite{Dahlberg-Dubbs-Newkirk-Tran}. Then to conclude the proof of Proposition \ref{proposition:p between 0 and 1 sharp form} one essentially compares the generalized curvature, which has to be constant, at the point $P$ and at the origin, which will  lead to a contradiction. 

We shall use the following notation: if $\Omega$ is simply connected, then $\gamma:[0,L]\to\delomega$ shall always denote a simple closed curve bounding $\Omega.$ Mostly we will also assume that 
\begin{equation}\label{eq:general choice of curve}
  |\gamma'|=1\quad\text{ and }
  \quad\nu=(\gamma'_2,-\gamma_1')\quad\text{ is the outward unit normal to }\delomega.
\end{equation}
This can always be achieved by reparametrization and chosing the orientation properly. The prime $(\cdot)'$ shall always denote the derivative with respect to the argument of $\gamma.$
If $|\gamma'|=1$ then the curvature $\kappa$ of $\delomega$ calculates as $\kappa(t)=\langle\gamma'(t),\nu'(t)\rangle.$

\begin{lemma}
\label{lemma:variational equation p-2 and p}
Let $p\in\re,$ $C>0$ and suppose that $\Omega$ is a $C^2$ minimizer of
\begin{equation}
 \label{lemma:first variation formula in n is 2}
  \inf\left\{\int_{\delomega}|x|^pd\sigma;\;|\Omega|=C\right\}.
\end{equation}
Assume $\gamma:[0,L]\to\partial\Omega$ is a simple closed curve. Then
\begin{equation}
 \label{eq:lemma var. method. p and p-2}
  p|\gamma(t)|^{p-2}\langle\gamma(t),\nu(t)\rangle+|\gamma(t)|^p \kappa(t)=\text{constant}\quad\forall\;t\in[0,L]
\end{equation}
if $0\notin\delomega.$ Whereas if $\gamma(t_0)=0$ for some $t_0\in [0,L],$ then \eqref{eq:lemma var. method. p and p-2} holds for all $t\in[0,L]\backslash\{t_0\}.$
\end{lemma}

\begin{remark}
The function $k= p|\gamma|^{p-2}\langle\gamma,\nu\rangle+|\gamma|^p \kappa$ is usually called the generalized curvature of $\delomega.$
\end{remark}

\begin{proof}
\textit{Step 1.}
Without loss of generality we assume the second case $0\in\delomega$ and assume that $\gamma(0)=\gamma(L)=0.$ We can also assume that \eqref{eq:general choice of curve} holds, since \eqref{eq:lemma var. method. p and p-2} is independent of the parametrization. Let $h\in C^{\infty}_c(0,L)$ be arbitrary and define the curve $g_s(t)=\gamma(t)+sh(t)\nu(t).$
For all $s$ small enough this is also a simple closed curve bounding a domain $\Omega_s$. Because $\Omega$ is a minimizer, we claim that 
\begin{equation}
 \label{eq:variational approach main condition p and p-2}
  \text{If }\quad\frac{d}{ds}\left[|\Omega_s|\right]_{s=0}=0\quad\text{ then }\quad \frac{d}{ds}\left[\int_{\delomega_s}|x|^pd\sigma\right]_{s=0}=0.
\end{equation}
Let us show \eqref{eq:variational approach main condition p and p-2}.
If this is not the case, this means that the function $h$ is such that the first equality in \eqref{eq:variational approach main condition p and p-2} holds true but
$$
  \frac{d}{ds}\left[\int_{\delomega_s}|x|^pd\sigma\right]_{s=0}\neq 0.
$$
Then define $\widehat{\Omega}_s=\lambda_s\Omega_s$ where $\lambda_s=\sqrt{{|\Omega|}}/\sqrt{{|\Omega_s|}}.$
Note that $|\widehat{\Omega}_s|=C$ for all $s$ and that, by the first iquality in \eqref{eq:variational approach main condition p and p-2}, $d/ds\,(\lambda_s)=0$ at $s=0.$ Hence we obtain that
$$
  \frac{d}{ds}\int_{\partial\widehat{\Omega}_s}|x|^pd\sigma\Bigg|_{s=0}= \frac{d}{ds}\left[\lambda_s^{p+1}\int_{\partial\Omega_s}|x|^pd\sigma\right]_{s=0}=  \frac{d}{ds}\left[\int_{\partial\Omega_s} |x|^pd\sigma\right]_{s=0}\neq 0.
$$
Thus for sufficiently small $s$ (negative or positive depending on the sign of the last inequality) the weighted perimeter of $\widehat{\Omega}_s$ is strictly smaller than that of $\Omega.$
This  contradicts the fact that $\Omega$ is a minimizer. 
\smallskip

\textit{Step 2.} We now calculate the derivatives in  \eqref{eq:variational approach main condition p and p-2} explcitly. First one gets
$$
  g_s=\gamma+sh\nu,\qquad g_s'=\gamma'+sh'\nu+sh\nu'.
$$
We therefore obtain
$$
  |\Omega_s|=\int_0^L(g_s)_1 (g_s)_2'dt=\int_0^L(\gamma_1+sh\nu_1) (\gamma_2'+sh'\nu_2+sh\nu_2'),
$$
which leads to, using partial integration to get rid of derivatives of $h$,
$$
  \frac{d}{ds}|\Omega_s|\Big|_{s=0}=\int_0^L(h'\gamma_1\nu_2+h\gamma_1\nu_2' +h\nu_1\gamma_2')=\int_0^Lh|\gamma'|^2=\int_0^Lh.
$$
On the other hand we have
$$
  \int_{\partial\Omega_s}|x|^pd\sigma=\int_0^L|g_s(t)|^p|g_s'(t)| dt.
$$
We therefor obtain that
$$
  \frac{d}{ds}\left[\int_{\delomega_s}|x|^pd\sigma\right]_{s=0}=A+B,
$$
where
\begin{align*}
 A=&\left[\int_0^L p|g_s(t)|^{p-1}\left(\frac{d}{ds}|g_s(t)|\right)|g_s'(t)| dt\right]_{s=0}
 \smallskip \\
 B=&\left[\int_0^L|g_s(t)|^p \frac{d}{ds}|g_s'(t)|\,dt\right]_{s=0}.
\end{align*}
Note that
$$
  |g_s|=\sqrt{\langle\gamma+sh\nu,\gamma+sh\nu\rangle}\quad\text{ and }\quad \frac{d}{ds}|g_s|\Big|_{s=0}=\frac{h\langle\gamma,\nu\rangle}{|\gamma|}.
$$
Thus we get
$$
  A=\int_0^Lp|\gamma|^{p-2}h\langle\gamma,\nu\rangle dt.
$$
Let us now calculate $B.$  As above
$$
  |g_s'|=\sqrt{\langle\gamma'+s h'\nu+sh\nu',\gamma'+sh'\nu+sh\nu'\rangle}.
$$
This leads to
$$
  \frac{d}{ds}|g_s'|\Big|_{s=0}=\frac{1}{|\gamma'|}(h'\langle\gamma',\nu\rangle +h\langle \gamma',\nu'\rangle)=h\langle\gamma',\nu'\rangle.
$$
Setting this into $B$ and adding $A+B$ we finally obtain that \eqref{eq:variational approach main condition p and p-2} 
 implies that 
$$
  \int_0^L \left(p|\gamma|^{p-2}\langle \gamma,\nu\rangle+|\gamma|^p\langle\gamma',\nu'\rangle\right)h\,dt=0\quad\forall\,h\in C_{c}^{\infty}(0,L)\text{ with }\int_0^L hdt=0.
$$
This implies the claim of the lemma.
\end{proof}

\begin{lemma}
\label{lemma:symmetry with respect to axis}
Let $p\in \re,$ $C>0$ and suppose that $\Omega$ is a $C^2$ minimizer of \eqref{lemma:first variation formula in n is 2}.
Assume $\gamma:[0,L]\to\partial\Omega$ is a simple closed curve with $|\gamma'|=1.$ Suppose there exists $t_0\in[0,L]$ such that
$$
  \gamma(t_0)\neq 0\quad\text{ and }\quad\frac{d}{dt}|\gamma(t)|\Big|_{t=t_0}=0.
$$
Then $\gamma$ is symmetric with respect to the line through $\gamma(t_0)$ and the origin.
\end{lemma}

\begin{proof} Without loss of generality we can assume (by rotation and reparametrization) that $t_0=0$ and $\gamma(0)=(\gamma_1(0),0)$ and $\gamma_1(0)>0.$ By hypothesis $\langle\gamma(0),\gamma'(0)\rangle =0$ and hence $\gamma_1'(0)=0.$ So using Lemma \ref{lemma:variational equation p-2 and p}, $\gamma$ satisfies the two equations
\begin{equation}
 \label{eq:raw version of ode. proof lemma}
  |\gamma'|^2=1\quad\text{ and }\quad p|\gamma|^{p-2} \langle\gamma,\nu\rangle+|\gamma|^p\langle\gamma',\nu'\rangle=k,
\end{equation}
for some constant $k.$ Deriving the first equation with respect to $t$ and multiplying with $|\gamma|^p$ we get that $|\gamma|^p(\gamma_1'\gamma_1''+\gamma_2'\gamma_2'')=0.$
Finally, multiplying this equation, respectively the second in \eqref{eq:raw version of ode. proof lemma}, with $\gamma_1'$ or $\gamma_2'$ and combining properly one easily gets that (using once more $|\gamma'|^2=1$)
\begin{align*}
 \gamma_1''=&\frac{\gamma_2'}{|\gamma|^p}\left(-k+p|\gamma|^{p-2}(\gamma_1\gamma_2'-\gamma_2\gamma_1')\right)=:F_1(\gamma,\gamma')
 \smallskip \\
 \gamma_2''=&\frac{\gamma_1'}{|\gamma|^p}\left(k-p|\gamma|^{p-2}(\gamma_1\gamma_2'-\gamma_2\gamma_1')\right)=:F_2(\gamma,\gamma').
\end{align*}
Thus, setting $\alpha=\gamma',$ we see that $(\gamma,\alpha)$ satisfies the initial value problem
\begin{align*}
 \gamma(0)=&(\gamma_1(0),0)\quad\text{ and }\quad \alpha(0)=(0,\gamma_2'(0))
 \smallskip \\
 &\left(\begin{array}{c}
        \gamma \\ \alpha
       \end{array}\right)'
       = \left(\begin{array}{c}
        \alpha \\ (F_1(\gamma,\alpha),F_2(\gamma,\alpha))
       \end{array}\right).
\end{align*}
The functions $F_1$ and $F_2$ are $C^1$ and hence Lipschitz as long as $\gamma\neq 0.$ They have the properties
$$
  F_1(\gamma_1,-\gamma_2,-\alpha_1,\alpha_2)=F_1(\gamma_1,\gamma_2,\alpha_1,\alpha_2),
  \qquad 
  F_2(\gamma_1,-\gamma_2,-\alpha_1,\alpha_2)=-F_2(\gamma_1,\gamma_2,\alpha_1,\alpha_2)
$$
Using these properties, it can be easily verified, by evaluating the differential equation at $-t,$ that also the curve $\omega(t)=(\gamma_1(-t),-\gamma_2(-t))$ satisfies the initial value problem. By uniqueness we obtain that $\omega(t)=\gamma(t).$ Since $(F_1,F_2)$ is locally Lipschitz, by the theorey of ordinary differential equations (see for instance \cite{Walter} page 68) the solution exists either for all times $t,$ it blows up or goes out of the region of definition of $(F_1,F_2).$ In the present case this means that the solution $\gamma$ exists and is unique for all times $t,$ unless $|\gamma|$ goes to $0$ or to $\infty.$ But it cannot go to infinity, because then the weighted perimiter would go to infinity and then $\Omega$ cannot be a minimizer. So we conclude, using that $\gamma$ is continuous by assumption, that $\omega(t)=\gamma(t)$ for all $t$ and this shows the claim of the lemma.  
\end{proof}

\smallskip

\begin{proof}[Proof of Proposition \ref{proposition:p between 0 and 1 sharp form}]
\textit{Step 1.} Let us show first that $\Omega$ has to be a simply connected set. If $\Omega$ is not connected, let us say $\Omega=\Omega_1\cup\Omega_2$ and $\Omega_1\cap\Omega_2=\emptyset$ for two nonempty sets, then choose $R_1,R_2>0$ such that $|\Omega_i|=\pi R_i^2.$ Using the hypothesis that there is equality in Proposition 4 and Theorem \ref{theorem:main theorem weighted isop. inequlity dim 2} part (i) for $\Omega_i$ we get
\begin{align*}
  (R_1^2+R_2^2)^{\frac{p+1}{2}}=&\left(\frac{|\Omega|}{\pi}\right)^{\frac{p+1}{2}}=\frac{1}{2\pi}\left(\int_{\delomega_1}|x|^pd\sigma+\int_{\delomega_2}|x|^pd\sigma\right)\geq R_1^{p+1}+R_2^{p+1}
  \smallskip \\
  >\left(R_1^2+R_2^2\right)^{\frac{p+1}{2}}
\end{align*}
In the last inequality we have assumed that $0<p<1$ and used that $a^s+b^s>(a+b)^s,$ if $0<s<1$ and $a,b>0.$ 
If $p=1$ there is still strict inequality because for one of the $i=1,2$ we must have $\int_{\delomega_i}|x|^p>2\pi R_i^{(p+1)/2}.$ If not, assume that we have shown Theorem \ref{theorem:main theorem weighted isop. inequlity dim 2} (iv) first for connected sets. Then both $\Omega_1$ and $\Omega_2$ would have to be balls centered at the origin, a contradiction to $\Omega_1\cap\Omega_2=\emptyset.$

Assume now that $\Omega$ is not simply connected and is of the form $\Omega=\Omega_0\backslash\overline{\Omega_1}$ for some $\Omega_1\subset\Omega.$
Then we obtain using Theorem \ref{theorem:main theorem weighted isop. inequlity dim 2} (i) that
$$
 \left(\frac{|\Omega|}{\pi}\right)^{\frac{p+1}{2}}<\left(\frac{|\Omega_0|}{\pi}\right)^{\frac{p+1}{2}}\leq \frac{1}{2\pi}\int_{\delomega_0}|x|^pd\sigma
 <\frac{1}{2\pi}\intdelomega|x|^pd\sigma=\left(\frac{|\Omega|}{\pi}\right)^{\frac{p+1}{2}}
$$
which is again a contradiction.
\smallskip

\textit{Step 2.} We now assume that $0\in\delomega$ and show that this leads to a contradiction. Let $\gamma:[0,L]\to\delomega$ be as in \eqref{eq:general choice of curve} and $\gamma(0)=\gamma(L)=0.$ Without loss of generality, by rotating the domain, we know by Lemma \ref{lemma:symmetry with respect to axis} that the maximum of $|\gamma|$ has to be achieved at $t=L/2,$ that for some $d>0$
$$
  \gamma\left(\frac{L}{2}\right)=(d,0)\quad\text{ and }\quad d=\max_{t\in[0,L]}|\gamma(t)|,
$$
and that $\gamma$ is symmetic with respect to the $x_1$ axis $\{(x_1,x_2)\in\re^2|\,x_2=0\}.$ Using this symmetry and the chosen orientation of $\gamma,$ one obtains that
$$
  \gamma'\left(\frac{L}{2}\right)=(0,1),\quad\;\gamma_1''\left(\frac{L}{2}\right) \leq 0\quad\text{ and hence}\quad \kappa\left(\frac{L}{2}\right)\geq 0.
$$
Note also that $\nu(L/2)=(1,0).$ Hence we obtain for the generalized curvature at $L/2$ that
\begin{equation}\label{eq:proof:k bigger zero in sharp form on boundary}
  k=p|\gamma|^{p-2}\langle\gamma;\nu\rangle+|\gamma|^p\kappa=pd^{p-2}d+d^p\kappa\geq pd^{p-1}>0.
\end{equation}

\textit{Step 3.} Using again Lemma \ref{lemma:symmetry with respect to axis} we obtain  that $\gamma_1(-t)=\gamma_1(t)$ and therefore $\gamma_1'(0)=0$ and $\gamma_2'(0)=-1.$ Therefore $\gamma_2$ is invertible near zero and $f=\gamma_1\circ\gamma_2^{-1}\in C^2((-\epsilon,\epsilon))$ for some $\epsilon>0.$ We use the new parametrization $\alpha$ of $\delomega$ near $0,$ given by
$t\mapsto\alpha(t)=(f(t),t),$ $t\in(-\epsilon,\epsilon).$
This new parametrization allows us to reduce the problem to a $1$-dimensional one, analyzing the function $f.$
Since $\alpha'=(f',1),$ $\alpha''=(f'',0)$ and keeping the reversed orientation in mind, we have the following formulas for the outer normal $\nu$ and curvature $\kappa$ 
$$
  \nu=\frac{(-1,f')}{\sqrt{1+f'^2}},\quad \kappa=-\frac{\alpha_1'\alpha_2''-\alpha_2'\alpha_1''}{|\alpha'|^3}=
  \frac{f''}{(1+f'^2)^{3/2}}\quad\text{ near }t=0.
$$
The generalized curvature $k=p|\alpha|^{p-2}\langle\alpha,\nu\rangle+|\alpha|^p\kappa$ near $t=0$ is
\begin{equation}
 \label{eq:proof:gen curvature for f par}
  k=\frac{p|\alpha(t)|^p}{\sqrt{1+f'(t)^2}}\frac{tf'(t)-f(t)}{|\alpha(t)|^2}+|\alpha(t)|^p\frac{f''(t)}{(1+f'(t)^2)^{3/2}}=m(t),
\end{equation}
where $m(t)$ is just an abbreviation for the right hand side.
\smallskip

\textit{Step 4.} We will show that
\begin{equation}
 \label{eq:proof:lim m of t is 0}
  \lim_{t\to 0}m(t)=0.
\end{equation}
This will be a contradiction to \eqref{eq:proof:k bigger zero in sharp form on boundary} and the fact that $k$ has to be constant, by Lemma \ref{lemma:variational equation p-2 and p}.
Using the facts that $p>0,$
$$
  \lim_{t\to 0}|\alpha(t)|=0,\quad \lim_{t\to 0}\sqrt{1+f'^2}=1,\quad\text{ and }\quad f\in C^2,
$$
to prove \eqref{eq:proof:lim m of t is 0}, it suffices to show that
\begin{equation}\label{eq:proof:tru for any f in C2}
  \frac{|tf'(t)-f(t)|}{|\alpha(t)|^2}=\frac{|tf'(t)-f(t)|}{(t^2+f(t)^2)}\quad\text{ is bounded on $(-\epsilon,\epsilon)$}.
\end{equation}
\eqref{eq:proof:tru for any f in C2} is true for any $f\in C^2([-\epsilon,\epsilon])$ with $f(0)=0.$ This can be seen in the following way. Since $f(0)=0,$ one has $\lim_{t\to 0}f(t)/t=f'(0)$ and
\begin{align*}
  \limsup_{t\to 0}\frac{|tf'(t)-f(t)|}{|t^2+f(t)^2|}=&\lim_{t\to 0}\frac{1}{|1+\frac{f(t)^2}{t^2}|}\limsup_{t\to 0}\left|\frac{t f'(t)-f(t)}{t^2}\right|
  \smallskip \\
  =&\frac{1}{1+f'(0)^2}\limsup_{t\to 0}\left|\frac{t f'(t)-f(t)}{t^2}\right|.
\end{align*}
It follows from de  l' Hopital rule that 
$$
  \lim_{t\to 0}\frac{t f'(t)-f(t)}{t^2}=\frac{f''(0)}{2},
$$
which proves that \eqref{eq:proof:tru for any f in C2} holds and the claim of Step 4.
\end{proof}

\section{Some results in general dimensions $n\geq 3$}
\label{general dimension}
In this section we shall use the following notations and abbreviations: let $n\in\mathbb{N}$
\begin{equation}
 \begin{split}
  \label{notations B S LL etc.}
  B_{\delta}^n(y)=&\left\{x\in\re^n|\, |x-y|<\delta\right\},\quad B_{\delta}^n=B_{\delta}^n(0),\quad \alpha_n=\LL^{n}(B_1^n) 
  \smallskip \\
  S_{\delta}^{n-1}=&\partial B_{\delta}^n= \left\{x\in\re^n|\, |x|=\delta\right\},\quad \omega_{n-1}=\HH^{n-1}(S^{n-1}_1).
\end{split}
\end{equation}
Recall that $\alpha_n=\omega_{n-1}/n.$ If the dimension is obvious we omit $n$ in the superscript e.g. $B_{\delta}=B_{\delta}^n$ and the same for $S_{\delta}^{n}.$ We also set $|\Omega|:=\LL^n(\Omega).$ The following is the main theorem, summarizing the results of \cite{Alvino Brock Mercaldo etc.}, \cite{Betta and alt.} and the present paper.

\begin{theorem}
\label{theorem:main if n bigger 3}
Let $n\geq 3$ and $\Omega\subset\re^n$ be a bounded open Lipschitz set. Regarding the inequality
\begin{equation}
 \label{eq:theorem:main isop. ineq. n bigger 3.}
     \left(\frac{n|\Omega|}{\omega_{n-1}}\right)^{\frac{n+p-1}{n}}\leq \frac{1}{\omega_{n-1}}\intdelomega|x|^pd\mathcal{H}^{n-1}.
\end{equation}
the following statements hold true:
\smallskip

(i) If $p\geq 0$ then \eqref{eq:theorem:main isop. ineq. n bigger 3.} holds true for all $\Omega.$
\smallskip 

(ii) If $-n+1<p<0,$ then we have for any $C>0$
$$
    \inf\left\{\int_{\delomega}|x|^p d\mathcal{H}^{n-1}:\,0\in\Omega\subset\re^n\text{ connected and }|\Omega|=C\right\}=0,
$$
where the infimum is taken over all bounded open smooth sets. In particular \eqref{eq:theorem:main isop. ineq. n bigger 3.} cannot hold for all $\Omega$ containing the origin.
\smallskip

(iii) If $p\leq -n+1,$ then \eqref{eq:theorem:main isop. ineq. n bigger 3.} holds true for all $\Omega$ containing the origin.
\smallskip

(iv) If there is equality in case (i) or (iii), then $\Omega$ must be a ball, centered at the origin if $p\neq 0.$
\end{theorem}

\begin{remark}
In (ii), if one omits the condition that $\Omega$ has to contain the origin, then it is trivial that the infimum is zero, by taking any domain with constant volume and shifting it to infinity. This is true in any dimension, in particular also if $n=2,$ compare with Theorem \ref{theorem:main theorem weighted isop. inequlity dim 2} (ii).
\end{remark}

\begin{proof}
Part (i) and its sharp form (iv) have been proven in \cite{Alvino Brock Mercaldo etc.} and \cite{Betta and alt.}, respectively the classical isoperimetric inequality. See also Proposition \ref{proposition:result of Alvino Brock et alt.} for the case $0<p<1$ and for starshaped domains. Part (ii) will be proven at the end of this section. Part (iii) and its sharp form (iv) will follow from Theorem \ref{theorem:p leq -n+1}, by setting $a(t)=t^p.$
\end{proof}
\smallskip

Although our results will not need the variational method, it should be mentioned, since it gives immediately some results on the nonvalidity of \eqref{eq:main isop. ineq. intro.} for the range $-n+2<p<0.$
It can be easily seen with this method that if one takes a ball centered at the origin and moves it slightly away in any direction, then the weighted perimeter decreases. It actually continues to decrease even more as one moves it further and further away. Moreover, for the range $-n+1<p<-n+2,$ Part (ii) of Theorem \ref{theorem:main if n bigger 3} gives an interesting example for a case when balls centered at the origin are local but not global minimizers of the weighted perimeter.  We shall illustrate first the variational method with an example.

\begin{example}
Fix $r>0$ and let $\Omega_s=B_r(0)+(s,0),$ where we will understand from now on $(s,0)=(s,0,\ldots,0)\in\re^n$ and $s\in \re.$ Obviously $|\Omega_s|=\alpha_nr^n$ for all $s$ and $\Omega_s$ contains the origin for all $s<r.$ We shall denote the weighted perimeter as
$$
  P\left(\Omega_s\right)=\int_{\delomega_s}|x|^pd\HH^{n-1}.
$$
Assume that $F=(F_1,F_2,\ldots,F_n)=(F_1,\tilde{F}):U\subset\re^{n-1}\to \partial B_r(0)$ is a parametrization of $\partial B_r(0)$ up to a set of $\HH^{n-1}$ measure $0.$ Let $g(u)du$ denote the surface element of $\partial B_r(0)$ in this parametrization, with $u\in U.$ Then $F+(s,0)=(F_1+s,\tilde{F})$ is a parametrization of $\partial \Omega_s$ and one obtains that
\begin{equation}
 \label{eq:first variation for general s}
  \frac{d}{ds}P(\Omega_s)=p\int_U\left(\left(F_1(u)+s\right))^2+|\tilde{F}(u)|^2\right)^{\frac{p}{2}-1}(F_1(u)+s)g(u)du.
\end{equation}
This gives
\begin{equation}
 \label{eq:example:ball is critical point}
  \frac{d}{ds}P(\Omega_s)\Big|_{s=0}=p\int_{\partial B_r}|x|^{p-2}x_1\dH=0\quad\text{ for any }p\in\re.
\end{equation}
Deriving \eqref{eq:first variation for general s} once more leads to
$$
  \frac{d^2}{ds^2}P\left({\Omega_s}\right)\Big|_{s=0}=p\left((p-2)r^{p-4}\int_{\partial B_r}x_1^2\dH+\HH^{n-1}(\partial B_r) r^{p-2}\right)
$$
Using that $\int_{\partial B_r}x_1^2=1/n\,\int_{\partial B_r}(x_1^2+\cdots+x_n^2)=r^2/n\, \HH^{n-1}(\partial B_r)$ the second variation simplifies to
$$
  \frac{d^2}{ds^2}P\left({\Omega_s}\right)\Big|_{s=0}=\frac{r^{p-2}}{n} \omega_{n-1}r^{n-1}\,p(p-2+n).
$$
One can now easily verify that if $n\geq 3$
\begin{equation}
 \label{eq:example:p values second variation}
  p(p-2+n)\text{ is }\left\{\begin{array}{rl}
                   >0 &\text{ if }-\infty<p<-n+2 
                   \\
                   =0 &\text{ if }p=-n+2 
                   \\
                   <0 &\text{ if }-n+2<p<0
                   \\
                   =0 &\text{ if }p=0
                   \\
                   >0 &\text{ if }p>0.
                  \end{array}\right.
\end{equation}
In view of \eqref{eq:example:ball is critical point} and the third inequality in \eqref{eq:example:p values second variation}, one obtains that $P(\Omega_s)$ decreases for increasing $s$ near the origin. 
This shows that \eqref{eq:theorem:main isop. ineq. n bigger 3.} cannot hold true if $-n+2<p<0.$
\end{example}

The equations \eqref{eq:example:p values second variation} remain true for much more general variations. We will not use them but nevertheless summarize the result in the present case. 
For the methods and proofs we refer to \cite{Rosales Canete Bayle Morgan}, \cite{do Carmo Barbosa} and \cite{Morgan Vari Formulae}. Let $\varphi_s:\re^n\to \re^n$ be a diffeomorphism for small $s.$ Define $\Omega_s=\varphi_s(B_r(0))$ and assume that the variaton is normal:
$\partial\varphi_s/\partial s=u$ is normal to $\delomega$ at $s=0.$ Moreover we assume that the variaton is such that it is volume preserving:
 $|\Omega_s|=|\Omega|$ for all $s.$ This implies, calculating the first variation of the volume
\begin{equation}\label{eq:volume preserving variation}
  \int_{\partial B_r}u\dH=0.
\end{equation}
Assume we are given a radial perimeter density $G:(0,\infty)\to (0,\infty),$ smooth in a neighborhood of $r,$ and $P_G$ is defined by
$$
  P_{G}(\Omega)=\int_{\delomega}G(|x|)\dH.
$$
Then the second variation for $P_G$ is given by
\begin{displaymath}
 \begin{split}
  &\frac{d^2}{ds^2}P_G(\Omega_s)\Big|_{s=0}=
  M_1+M_2\, \quad\text{ where }\quad
  \smallskip 
  \\
  &M_1=G(r)\left(\int_{\partial B_r}|\nabla_{\sum}u|^2-\frac{n-1}{r^2}\int_{\partial B_r}u^2\right),
  \quad M_2=\left((n-1)\frac{G'(r)}{r}+G''(r)\right)\int_{\partial B_r}u^2
\end{split}
\end{displaymath}
where $\nabla_{\sum}u$ is the covariant derivative along $\partial B_r$. By the standard isoperimetric inequality one has that $M_1\geq 0$ for any $u$ satisfying \eqref{eq:volume preserving variation}. This follows again by variational methods (taking $G=1$), or equivalently, it is the Poincar\'e inequality on the sphere with optimal constant, see \cite{do Carmo Barbosa} Example (2.13). Moreover one can show by the first variation of $P_G$ that balls centered at the origin are always critical points of $P_G$ under the volume constraint. So this implies that $B_r$ can be a minimizer of $P_G$ if and only if $M_2\geq 0,$ i.e.
$$
  (n-1)\frac{G'(r)}{r}+G''(r)\geq 0
$$  
In the present case $G(r)=r^p,$ so the last inequality holds true if and only if $p(n+p-2)\geq 0.$ This is precisely again \eqref{eq:example:p values second variation}. Note that if $-n+1<p<-n+2,$ then $p(n+p-2)>0,$
 but nevertheless Theorem \ref{theorem:main if n bigger 3} (ii) shows that balls centered at the origin are not global minimizers.
\smallskip

We shall now prove Part (iii) of Theorem \ref{theorem:main if n bigger 3}.
However this can be done for much more general densities and we state the result in that form. It follows the same idea as in the proof of the corresponding result when $p\geq 1$ in \cite{Betta and alt.}.

\begin{theorem} 
\label{theorem:p leq -n+1}
Let $n\geq 2$ and $a:(0,\infty)\to [0,\infty)$ be a continuous function such that
\begin{equation}
 \label{eq:theorem:decreasing and convex}
  t\mapsto a\left(t^{\frac{1}{n}}\right) t^{\frac{n-1}{n}}\quad\text {is non-increasing and convex}
\end{equation}
Then any bounded open Lipschitz set $\Omega\subset\re^n$ containing the origin satisfies
\begin{equation}\label{eq:theorem:a convex decreasing main ineq}
  \int_{\partial B_R}a(|x|)d\mathcal{H}^{n-1}(x)\leq \intdelomega a(|x|)d\mathcal{H}^{n-1}(x),
\end{equation}
where $B_R$ is the ball of radius $R$ centered at the origin and with the same volume as $\Omega.$ If there is equality in \eqref{eq:theorem:a convex decreasing main ineq} and $a(t)>0$ for all $t\in(0,\infty),$ then $\Omega$ must be a ball centered at the origin.
\end{theorem}

\begin{remark}
\label{remark:a itself non increasing}
The hypothesis of the theorem implies that $a$ itself has to be non-increasing.
\end{remark}

\begin{figure}[h]
\begin{center}
\label{figure Omega i and r ij}
\includegraphics[scale=0.45]{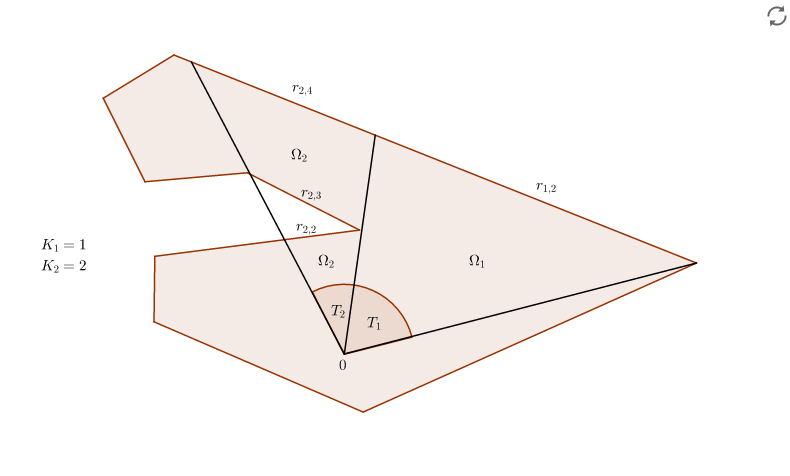}
\end{center}
\caption{Construction of  $\Omega_i$ and $r_{i,j}$}
\end{figure}

\begin{proof}
\textit{Step 1.}
Let $Q=(0,\pi)^{n-2}\times (0,2\pi)$ and $H:Q\to S^{n-1}$ denote the hypershperical coordinates and $g(\varphi)d\varphi$ the surface element of $S^{n-1}$ in these coordinates, see Appendix, Section \ref{Appendix}. By an approximation argument, using that $0\notin\delomega$ and $a$ continuous, we can assume that $\Omega$ is of the following form (this is exactly the same as in \cite{Betta and alt.} proof of Theorem 2.1, Step 3): there exists $\Omega_i,$ $i=1,\ldots,l$ such that $\Omegabar=\bigcup_{i=1}^l\Omega_i$,
and each $\Omega_i$ is given in the following way 
(see Figure 1 giving an idea how to construct $\Omega_i$). We assume that there exists disjoint open subsets $T_i\subset Q,$ such that
$$
  \overline{Q}=\bigcup_{i=1}^l\overline{T}_i\,,
$$
and for each $i=1,\ldots,l$ there exists $K_i\in\mathbb{N}$ and functions $r_{i,s}$ for $s=1,\ldots,2K_i$
$$
  r_{i,s}\in C^1\left(\overline{T}_i\right),\quad 0=r_{i,1}<r_{i,2}< \ldots <r_{i,2K_i},
$$
such that
$$
  \Omega_i=\left\{rH(\varphi)|\;\varphi\in \overline{T}_i\,,\quad r_{i,2s-1}(\varphi)\leq r\leq r_{i,2s}(\varphi)\quad\text{ for }s=1,\ldots,K_i\right\}.
$$
We can assume that $r_{i,1}=0$ because $0\in\Omega.$ Let $G$ denote the polar coordinates $G(t,\varphi)=tH(\varphi)$ (see Appendix) whose Jacobian determinant is given by $t^{n-1}g(\varphi).$ Therefore we obtain that 
\begin{align*}
 |\Omega_i|=&\int_{T_i}\sum_{s=1}^{K_i}\int_{r_{i,2s-1}(\varphi)}^{r_{i,2s}(\varphi)}t^{n-1}g(\varphi)dt\,d\varphi 
 =\int_{T_i}\sum_{s=1}^{K_i}\frac{1}{n}\left(r_{i,2s}^n-r_{i,2s-1}^n\right)g(\varphi)d\varphi 
 \smallskip \\
 \geq & \int_{T_i}\frac{1}{n}r_{i,2}^n(\varphi)g(\varphi)d\varphi.
\end{align*}

Let us define a function $r:Q\to\re,$ $\mathcal{H}^{n-1}$ almost eveywhere in $Q,$ by
$$
  r:\bigcup_{i=1}^l T_i\to \re,\quad r(\varphi)=r_{i,2}(\varphi)\quad\text{if }\varphi\in T_i\,.
$$
For $p\in S^{n-1}$ we set $\tilde{r}(p)=r(H^{-1}(p)),$ and thus $\tilde{r}(H(\varphi))=r(\varphi)$ almost everywhere in $Q.$ In this way we obtain that
\begin{equation}
 \label{eq:estimate for vol Omega with tilde r}
  |\Omega|=\sum_{i=1}^l|\Omega_i|\geq \frac{1}{n}\int_{Q}r(\varphi)^ng(\varphi)d\varphi = \frac{1}{n}\int_{S^{n-1}}\tilde{r}^nd\mathcal{H}^{n-1}.
\end{equation}

\textit{Step 2.} We now estimate the weighted perimeter. For $i=1,\ldots,l$ us define $\Gamma_i$ by
$$ 
  \Gamma_i=\bigcup_{s=2}^{2K_i}\left\{r_{i,s}(\varphi)H(\varphi)|\; \varphi\in \overline{T}_i\right\}.
$$
In this way $\bigcup_{i=1}^l\Gamma_i\subset\delomega$ and we obtain that
$$
  \intdelomega a(|x|)d\mathcal{H}^{n-1}(x)\geq\sum_{i=1}^l\int_{\Gamma_i} a(|x|) d\mathcal{H}^{n-1}(x)=
  \sum_{i=1}^l\sum_{s=2}^{2K_i}\int_{T_i}a(r_{i,s}(\varphi))\tilde{g}_{i,s} (\varphi)d\varphi,
$$
where $\tilde{g}_{i,s}$ is given by
\begin{equation}\label{eq:proof:gtilde i s definition}
  \tilde{g}_{i,s}(\varphi)=\sqrt{\det\left(\left\langle\frac{\partial F_{i,s}}{\partial \varphi_j},\frac{\partial F_{i,s}}{\partial \varphi_k}\right\rangle\right)_{j,k=1,\ldots, n-1}} 
\end{equation}
with the parametrizations $F_{i,s}(\varphi)=r_{i,s}(\varphi)H(\varphi).$ Using Lemma \ref{lemma:parametrization r H} and the fact that $a\geq 0,$ we obtain that 
\begin{equation}\label{eq:proof:inequality by Lemma 11 for tilde g}
  a\,\tilde{g}_{i,s}\geq a\, r_{i,s}^{n-1} g \quad\text{ for all $i,s.$}
\end{equation}
It follows that
$$
  \intdelomega a(|x|) d\mathcal{H}^{n-1}(x)
  \geq 
  \sum_{i=1}^l\sum_{s=2}^{2K_i}\int_{T_i}a(r_{i,s}(\varphi)) r_{i,s}^{n-1}(\varphi) g (\varphi)d\varphi
  \geq
  \sum_{i=1}^l\int_{T_i}a(r_{i,2}(\varphi)) r_{i,2}^{n-1}(\varphi) g (\varphi)d\varphi
$$
Using again the definition of $r$ and $\tilde{r}$ introduced at the end of Step 1 we obtain 
\begin{equation}
 \label{eq:proof:estimate of y to p with r tilde}
  \intdelomega a(|x|)d\mathcal{H}^{n-1}(x)
  \geq \int_{S^{n-1}}a(\tilde{r})\tilde{r}^{n-1}d\mathcal{H}^{n-1}.
\end{equation}

\textit{Step 3.} 
Let us define $h$ by
$$
  h(t)=a\left(t^{\frac{1}{n}}\right) t^{\frac{n-1}{n}}.
$$
It now follows from \eqref{eq:estimate for vol Omega with tilde r}, from the fact that $h$ is decreasing, from Jensen inequality and finally 
\eqref{eq:proof:estimate of y to p with r tilde}, that
\begin{equation}
\label{eq:proof:h ans omega n-1 in Step 3}
\begin{split}
  h\left(\frac{n|\Omega|}{\omega_{n-1}}\right)\leq & h
   \left(\frac{1}{\omega_{n-1}}\int_{S^{n-1}}\tilde{r}^nd\mathcal{H}^{n-1}\right)
  \leq 
  \frac{1}{\omega_{n-1}}\int_{S^{n-1}}h\left(\tilde{r}^n\right)d\mathcal{H}^{n-1}
  \smallskip \\
  =&\frac{1}{\omega_{n-1}}\int_{S^{n-1}}a(\tilde{r})\tilde{r}^{n-1} d\mathcal{H}^{n-1}\leq \frac{1}{\omega_{n-1}}\intdelomega a(|x|)d\mathcal{H}^{n-1}(x).
 \end{split}
\end{equation}
Since $|\Omega|=|B_R|$ one easily verifies that 
\begin{equation}
 \label{eq:proof:R and Omega volume connection}
  R=\left(\frac{n|\Omega|}{\omega_{n-1}}\right)^{\frac{1}{n}}\quad\text{ and }\quad h\left(\frac{n|\Omega|}{\omega_{n-1}}\right)=a(R)R^{n-1}=\frac{1}{\omega_{n-1}}\int_{\partial B_R}a(|x|)d\mathcal{H}^{n-1}(x).
\end{equation}
Plugging this identity into  \eqref{eq:proof:h ans omega n-1 in Step 3} concludes the proof inequality \eqref{eq:theorem:a convex decreasing main ineq}.
\smallskip

\textit{Step 4.}
We now deal with the case of equality in \eqref{eq:theorem:a convex decreasing main ineq}, first with the additional assumption that $\Omega$ is a $C^1$ set starshaped with respect to the origin: i.e. $\delomega$ can be parametrized (up to a set of $\mathcal{H}^{n-1}$ measure zero) as $\delomega=\{r(\varphi)H(\varphi)|\,\varphi\in Q\}.$ In other words $l=1,$ $K_1=1$ and $r=r_{1,2}\in C^1(Q).$ By hypothesis we must have equality in \eqref{eq:estimate for vol Omega with tilde r}. Let us show that $r$ is constant. If $r$ is not constant, then by Lemma \ref{lemma:parametrization r H} there exists $\varphi_0$ and  a neighborhood of $U$ in $Q$ containing $\varphi_0$ such that, recalling \eqref{eq:proof:gtilde i s definition}, 
$$
  \tilde{g}_{1,1}>r^{n-1}g\quad\text{ in }U.
$$
Thus we obtain a strict inequality in \eqref{eq:proof:estimate of y to p with r tilde} and we cannot have equality in \eqref{eq:theorem:a convex decreasing main ineq}. (Note that if we additionally assume that $h$ is strictly convex, then equality in Jensen inequality also implies that $\tilde{r}$ must be constant.) We thus obtain that $\Omega$ is a ball centered at the origin with radius $r.$

\smallskip

\textit{Step 5.} Let us treat now the general case of equality in \eqref{eq:theorem:a convex decreasing main ineq}. Since $\Omega$ is bounded there exsist $M>0$ such that 
$$
  |x|\leq M\quad \text{ for all }x\in \delomega\text{ and all }x\in \partial B_R\,.
$$
Let us define $\tilde{a}$ by
$$
  \tilde{a}(t)=\left\{\begin{array}{rl}
                    a(t)-a(M)&\text{ if }0<t\leq M 
                    \smallskip \\
                    0&\text{ if }t\geq M.
                   \end{array}\right.
$$
Note that $\tilde{a}$ is continuous and $\tilde{a}\geq 0,$ because $a$ is non-increasing (Remark \ref{remark:a itself non increasing}). Moreover $\tilde{h}$ defined by
$$
  \tilde{h}(t)=\tilde{a}\big(t^{\frac{1}{n}}\big)t^{\frac{n-1}{n}}=
  \left\{\begin{array}{rl}
  a\big(t^{\frac{1}{n}}\big)t^{\frac{n-1}{n}}-a(M)t^{\frac{n-1}{n}} & \text{if }0<t^{\frac{1}{n}}\leq M
  \smallskip \\
  0 &\text{if }t^{\frac{1}{n}}\geq M
  \end{array}\right.
$$
is non-increasing and convex, because it is the sum of two nonincreasing and convex functions. Therefore we can apply \eqref{eq:theorem:a convex decreasing main ineq} to obtain that
$$
  \intdelomega\tilde{a}(|x|)d\mathcal{H}^{n-1}(x)\geq \int_{\partial B_R}\tilde{a}(|x|)d\mathcal{H}^{n-1}(x).
$$
By the choice of $M$ and the definition of $\tilde{a}$ we have that $\tilde{a}(|x|)=a(|x|)-a(M)$ for any $x\in \delomega\cup\partial B_R$. So the inequality becomes
$$
  \intdelomega a(|x|)d\mathcal{H}^{n-1}(x)-a(M)\mathcal{H}^{n-1}(\delomega)\geq \int_{\partial B_R} a(|x|)d\mathcal{H}^{n-1}(x) -a(M)\mathcal{H}^{n-1}(\partial B_R).
$$ 
Using now the assumption that we have equality in \eqref{eq:theorem:a convex decreasing main ineq} and that $a(M)>0$, we obtain that $\mathcal{H}^{n-1}(\delomega)\leq \mathcal{H}^{n-1}(\partial B_R).$ Which implies, in view of the classical isoperimetric inequality, that $\Omega$ must be a ball. Since $0\in\Omega$ we are in the case of the assumptions of Step 4 and the result follows. 
\end{proof}

\smallskip We will now prove Part (ii) of the main theorem.
We will use the notation: if $x\in \re^n$ write $x=(x_1,x'),$ where $x'\in \re^{n-1}.$
\smallskip 

\begin{proof}[Proof of Theorem \ref{theorem:main if n bigger 3} Part (ii).] 
\textit{Step 1.}
It is sufficient to find  a sequence of sets $\Omega_{\epsilon}$ such that each $\Omega_{\epsilon}$ is Lipschitz, bounded, connected, $0\in\Omega_{\epsilon}$ and
\begin{equation}
 \label{counterexample with limit}
  \lim_{\epsilon\to 0}|\Omega_{\epsilon}|\to C>0\quad\text{ but }\quad \lim_{\epsilon\to 0} \int_{\delomega_{\epsilon}}|x|^p \dH=0.
\end{equation}
Then, by approximation, \eqref{counterexample with limit} holds also for a sequence of smooth sets. Then this sequence can be rescaled, defining a new sequence $\widehat{\Omega}_{\epsilon}=\lambda_{\epsilon}\Omega_{\epsilon}$ with $\lambda_{\epsilon}=C^{1/n}|\Omega_{\epsilon}|^{-1/n}$ so that  $|\widehat{\Omega}_{\epsilon}|$ is constant and $\widehat{\Omega}_{\epsilon}$ still satisfies the second limit in \eqref{counterexample with limit}. 

Let us now show \eqref{counterexample with limit}. Again by a rescaling argument, we can fix one $C$ and assume without loss of generality that $C=\alpha_nR^n$ for some $R>0.$
The idea is to choose $\Omega_{\epsilon}$ as a ball of radius $R$ and center going away to infinity, with a long and narrow cylinder attached to it so that it contains the origing, see Figure \ref{figura}. 
One has to choose the length and radius of the cylinder carefully. More precisely:
let us fix some $R>0$ define 
$$
  c_{\epsilon}=\sqrt{R^2-\epsilon^4}+\epsilon^{-n+1}.
$$
Then $\Omega_{\epsilon}$ shall be defined as the union of the following three sets: $\Omega_{\epsilon}=A_{\epsilon}\cup M_{\epsilon}\cup D_{\epsilon},$
\begin{align*}
 A_{\epsilon}=&\left\{ x\in B_{\epsilon^2}^n:\, x_1\leq 0\right\}
 \smallskip   \\
 M_{\epsilon}=&\left(0,\epsilon^{-n+1}\right)\times B_{\epsilon^2}^{n-1}=\left\{x\in \re^n:\, 0\leq x_1\leq \epsilon^{-n+1},\quad |x'|<\epsilon^2\right\}
 \smallskip \\
 D_{\epsilon}=&\left\{x\in \re^n:\, \left|x-\left(c_{\epsilon},0\right)\right|<R,\quad x_1\geq \epsilon^{-n+1}\right\}
\end{align*}

\begin{figure}[h]
\begin{center}
\includegraphics[scale=0.35]{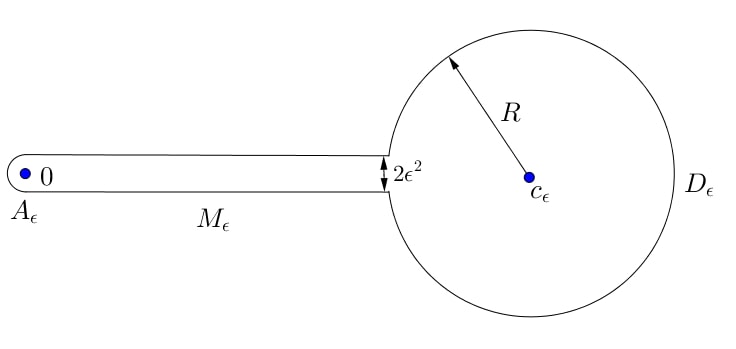}
\end{center}
\caption{the domain $\Omega_{\epsilon}$}
\label{figura}
\end{figure}

By construction $\Omega_{\epsilon}$ is a bounded connected Lipschitz set containing the origin. Note that
$$
  |A_{\epsilon}|=\frac{1}{2}\alpha_n\epsilon^{2n},\quad  |M_{\epsilon}|=\frac{\alpha_{n-1}\epsilon^{2(n-1)}}{\epsilon^{n-1}}=\alpha_{n-1}\epsilon^{n-1}.
$$
So we obtain that $\lim_{\epsilon\to 0}|\Omega_{\epsilon}|=\lim_{\epsilon\to 0}|D_{\epsilon}|=\alpha_nR^n.$
\smallskip 

\textit{Step 2.} It remains to estimate the weighted perimeter of $\delomega_{\epsilon}.$ The contribution coming from $A_{\epsilon}$ is
$$
  \frac{1}{2}\int_{\partial B_{\epsilon^2}^n}|x|^p\dH=\frac{\omega_{n-1}}{2}\epsilon^{2(n-1)+2p}\to 0\quad\text{ for }\epsilon \to 0,
$$
because $n+p-1>0.$ Note that for any $x\in \partial D_{\epsilon},$ we have the estimates, (since $p<0$), $|x|\geq \epsilon^{-n+1}$ and hence $|x|^p\leq \epsilon^{-p(n-1)}.$ This gives the estimate
$$
  \int_{\partial D_{\epsilon}}|x|^p \dH\leq \epsilon^{-p(n-1)}\omega_{n-1}R^{n-1}\to 0\quad \text{ for }\epsilon \to 0.
$$
It remains to estimate the contribution coming from $M_{\epsilon}.$ It is equal to 
$$
  \int_0^{\epsilon^{-n+1} }\int_{S_{\epsilon^2}^{n-2}}\left(x_1^2+|x'|^2\right)^{\frac{p}{2}}d\HH^{n-2}(x')dx_1=
  \omega_{n-2}\epsilon^{2(n-2)}\int_0^{\epsilon^{-n+1} }\left(x_1^2+\epsilon^4\right)^{\frac{p}{2}}dx_1
  \smallskip \\
  =
  L_1+L_2\,,
$$
where
$$
  L_1=\omega_{n-2}\epsilon^{2(n-2)}\int_0^{\epsilon^2}\left(x_1^2+\epsilon^4\right)^{\frac{p}{2}}dx_1\,\text{ and }\,
  L_2=\omega_{n-2}\epsilon^{2(n-2)}\int_{\epsilon^2}^{\epsilon^{-n+1}}\left(x_1^2+\epsilon^4\right)^{\frac{p}{2}}dx_1\,.
$$
To estimate $L_1$ we just use that $x_1^2+\epsilon^4\geq \epsilon ^4,$ and that $p<0$ to get $L_1\leq \omega_{n-2}\epsilon^{2(n+p-1)}\to 0.$ To estimate $L_2$ we use that $x_1^2+\epsilon^4\geq x_1^2$ and therefore
$$
  L_2\leq \omega_{n-2}\epsilon^{2(n-2)}\int_{\epsilon^2}^{\epsilon^{-n+1}}x_1^pdx_1\,.
$$
we distinguish 3 cases.
\smallskip 

\textit{Case 1, $p>-1$.} In this case $p+1>0$ and we obtain by explicit integration
$$
  L_2\leq \frac{\omega_{n-2}}{p+1}\epsilon^{2(n-2)}\left(\epsilon^{-(n-1)(p+1)}-\epsilon^{2(p+1)}\right)\leq \frac{\omega_{n-2}}{p+1}\epsilon^{2(n-2)-(n-1)(p+1)}.
$$
So $L_2$ tends to zero as $\epsilon\to 0$ if and only if $2(n-2)-(n-1)(p+1)>0,$ which is equivalent to
$$
  p<\frac{n-3}{n-1}.
$$
This is true, because $p<0$ and $n\geq 3.$
\smallskip

\textit{Case 2, $p=-1$.} Explicit integration as in Case 1 gives that
$$
  L_2\leq -\omega_{n-2}(n+1)\epsilon^{2(n-2)}\log(\epsilon)\to 0,
$$
because $n\geq 3.$
\smallskip

\textit{Case 3, $p<-1$.} Hence $p+1<0$ and we obtain
$$
  L_2\leq -\frac{\omega_{n-2}}{p+1}\epsilon^{2(n-2)}\left(\epsilon^{2(p+1)}-\epsilon^{-(n-1)(p+1)}\right)\leq 
  -\frac{\omega_{n-2}}{p+1}\epsilon^{2(n-2)+2(p+1)}\to 0,
$$
using again that $n+p-1>0.$
\end{proof}

\begin{remark}
Note that also Case 3 cannot occur if $n=2.$
\end{remark}

We give here a proof of Theorem \ref{theorem:main if n bigger 3} (i) in the special case of starshaped domains, since it is quite short and contains a new and interesting interpolation argument due to Alvino et alt. \cite{Alvino Brock Mercaldo etc.}. The general case follows from a more standard, but weighted, symmetrization argument, see also \cite{Alvino Brock Mercaldo etc.}.

\begin{proposition}
\label{proposition:result of Alvino Brock et alt.}
Let $n\geq 3,$ $0< p< 1,$  $\Omega$ be a bounded open domain, starshaped with respect to the origin and with defining function $R$ (see \eqref{definition:starshaped domain} in Section \ref{Appendix}) such that $R$ is Lipschitz. Then it holds that
$$
  n\alpha_n^{\frac{1-p}{n}}|\Omega|^{\frac{p+n-1}{n}}\leq \intdelomega |x|^p \dH.
$$
Moreover, if there is equality then $\Omega$ has to be a ball centered at the origin.
\end{proposition}

\begin{proof} \textit{Step 1.} Let us define by $P_p(\Omega)$ the right hand side of the inequality in the proposition. We abbreviate $S^{n-1}=S_1^{n-1},$ see \eqref{notations B S LL etc.} for notation.
From Lemma \ref{lemma1 for starshaped} in the appendix we obtain that
$$
  P_p(\Omega):=\intdelomega |x|^p\dH=\intS R^{p+n-1}\sqrt{1+\frac{1}{R^2}\left|\nablas R\right|^2}\,\dH.
$$
Define 
$$
  Z=R^{\frac{p+n-1}{n-1}}\quad\text{ and }\quad A=\left(\frac{n-1}{p+n-1}\right)^2\leq 1,
$$
Note that for any $n\geq 3$ it holds that
\begin{equation}
 \label{eq:2 inequ in p in 0 1}
  1\geq A=\left(\frac{n-1}{p+n-1}\right)^2\geq \left(\frac{3-1}{p+3-1}\right)^2\geq 1-p\geq 0\quad\text{ for all }p\in [0,1].
\end{equation}
From Lemma \ref{lemma 2 for starshaped} and the previous inequality we obtain
\begin{equation}
 \label{eq:1 inequality in p in 0 1}
 \begin{split}
  P_p(\Omega)=&\intS Z^{n-1}\sqrt{1+\frac{1}{Z^2}\left|\nablas Z\right|^2 A}\,\dH
  \smallskip 
  \\
  \geq & \intS Z^{n-1}\sqrt{1+\frac{1}{Z^2}\left|\nablas Z\right|^2 (1-p)}\,\dH,
 \end{split}
\end{equation}
Now use that for any $a,b\geq 0$ the mapping $t\mapsto\log\left[\int\sqrt{a+bt}\right]$ is concave, so that  
$$
  \int\sqrt{a+b((1-p) t_1+p t_2)}\geq  \left(\int\sqrt{a+b t_2}\right)^{p} \left(\int\sqrt{a+b t_1}\right)^{1-p},\quad t_1,t_2\geq 0.
$$
Using this it follows from \eqref{eq:1 inequality in p in 0 1} (with $t_1=1,t_2=0)$
\begin{equation}
 \label{eq:3 ineq in p in 0 1}
  P_p(\Omega)\geq \left(\intS Z^{n-1}\dH\right)^{p}
   \left(\intS Z^{n-1}\sqrt{1+\frac{1}{Z^2}\left|\nablas Z\right|^2}\,\dH\right)^{1-p}
\end{equation}
Define a new domain, also starshaped with respect to the origin,
$$
  \Omega^Z=\{0\}\cup\left\{x\in\re^n\backslash\{0\}\,:\, \frac{x}{|x|}=\theta\in S^{n-1},\,0<|x|<Z(\theta)\right\}.
$$
It follows from Lemma \ref{lemma1 for starshaped}, the standard isoperimetric inequality and \eqref{eq:volume of starshaped domain} that
\begin{align*}
  \intS Z^{n-1}&\sqrt{1+\frac{1}{Z^2}\left|\nablas Z\right|^2}\,\dH =\mathcal{H}^{n-1}(\partial \Omega^Z)\geq \LL^n(\Omega^Z)^{\frac{n-1}{n}}n\alpha_n^{\frac{1}{n}} 
  \smallskip 
  \\
  =&\left(\frac{1}{n}\intS Z^n \dH\right)^{\frac{n-1}{n}}\alpha_n^{\frac{1}{n}} n
  =\left(\intS Z^n \dH\right)^{\frac{n-1}{n}}(n\alpha_n)^{\frac{1}{n}}.
\end{align*}
We plug this into the \eqref{eq:3 ineq in p in 0 1} and use the definition of $Z$ to conclude that
\begin{equation}
 \label{eq:7 in my notes p 01}
  P_p(\Omega)\geq \left(\intS R^{p+n-1}\dH\right)^{p}\left(\intS R^{\frac{(p+n-1) n}{n-1}}\dH\right)^{\frac{(1-p)(n-1)}{n}}(n\alpha_n)^{\frac{1-p}{n}}.
\end{equation}
The idea is to use now H\"older inequality in the form
\begin{equation}
 \label{eq:8 in my notes Hoelder inequality}
  \intS R^n=\intS R^{np}R^{n-np}\leq \left(\intS R^{np q}\right)^{\frac{1}{q}} \left(\intS R^{(n-np)q'}\right)^{\frac{1}{q'}}
\end{equation}
with $1<q<\infty,$
$$
  q=\frac{p+n-1}{pn},\quad q'=\frac{p+n-1}{(n-1)(1-p)},\quad\frac{1}{q}+\frac{1}{q'}=1.
$$
We take \eqref{eq:8 in my notes Hoelder inequality} to the power $(p+n-1)/n$ and multiply by $(n\alpha_n)^{(1-p)/n}$ to get (using \eqref{eq:7 in my notes p 01} in the last step)
\begin{align*}
  (n\alpha_n)^{\frac{1-p}{n}}&\left(\intS R^n\dH\right)^{\frac{p+n-1}{n}}
  \smallskip
  \\
  \leq& (n\alpha_n)^{\frac{1-p}{n}}\left(\intS R^{p+n-1}\dH\right)^{\frac{p+n-1}{n q}} \left(\intS R^{\frac{n(p+n-1)}{n-1}}\dH\right)^{\frac{p+n-1}{n q'}}
  \smallskip 
  \\
  =&(n\alpha_n)^{\frac{1-p}{n}}\left(\intS R^{p+n-1}\dH\right)^{p} \left(\intS R^{\frac{n(p+n-1)}{n-1}}\dH\right)^{\frac{(1-p)(n-1)}{n}}
  \smallskip
  \\
  \leq& P_p(\Omega).
\end{align*}
Thus it follows from  \eqref{eq:volume of starshaped domain} that 
$$
  n\alpha_n^{\frac{1-p}{n}}|\Omega|^{\frac{p+n-1}{n}}
  =
  (n\alpha_n)^{\frac{1-p}{n}}\left(\intS R^n\dH\right)^{\frac{p+n-1}{n}}
  \leq
  P_p(\Omega),
$$
which proves the first part of the proposition.
\smallskip

\textit{Step 2.} Let us consider the case of equality. In that case there must be equality in \eqref{eq:8 in my notes Hoelder inequality}, which is only possible if for some constant $c\in \re$
$$
  R(\theta)^{p+n-1}=c R(\theta)^{\frac{n(p+n-1)}{n-1}}\quad\text{ for all }\theta\in S^{n-1}.
$$
This is only possible if $R$ is constant.
\end{proof}

\begin{remark}
The present proof does not work for $n=2,$ since in that case \eqref{eq:2 inequ in p in 0 1} is not satisfied for all $p\in (0,1).$
\end{remark}

\section{Appendix: hyperspherical coordinates in $\re^n.$}
\label{Appendix}

We have used in Section \ref{general dimension} the explicit form of hyperspherical coordinates and their properties. Let us define $Q\subset\re^{n-1}$ by
$$
  Q=(0,\pi)^{n-2}\times (0,2\pi).
$$
The hypershperical coordinates  $H=(H_1,\ldots,H_n):Q\to S^{n-1}\subset\re^n$ are defined as: for $k=1,\ldots,n$ and $\varphi=(\varphi_1,\ldots,\varphi_{n-1})$
$$
  H_k(\varphi)=\cos\varphi_k\prod_{l=0}^{k-1}\sin\varphi_l\quad\text{ with convention }\varphi_0=\frac{\pi}{2},\quad\varphi_n=0.
$$
A calculation shows that
$
  d_i(\varphi):=\left\langle \frac{\partial H}{\partial \varphi_i},\frac{\partial H}{\partial \varphi_i}\right\rangle =\prod_{l=0}^{i-1}\sin^2\varphi_l>0.
$
One verifies that the metric tensor $S^{n-1}$ in these coordinates is diagonal 
$
  g_{ij}(\varphi)=\left\langle\frac{\partial H}{\partial \varphi_i},\frac{\partial H}{\partial \varphi_j}\right\rangle=\delta_{ij}d_i(\varphi),
$
($\delta_{ij}=1$ if $i=j$ and $0$ else)
and hence the surface element $g$ is given by
$$
  g(\varphi)=\sqrt{\det g_{ij}(\varphi)}=\prod_{k=1}^{n-1}d_k(\varphi)= \prod_{k=1}^{n-2}(\sin\varphi_k)^{n-k-1}.
$$
Let us also denote the polar coordinates in $\re^n,$ denoted as $G:(0,\infty)\times Q\to\re^n,$ given by
$
  G(t,\varphi)=tH(\varphi).
$
Its Jacobian determinant is then given by
\begin{equation}
 \label{eq:jacobian determinant of polar coor}
  \det DG(t,\varphi)=t^{n-1}g(\varphi).
\end{equation}
For our purpose the following lemma will be useful.

\begin{lemma}\label{lemma:parametrization r H}
Suppose $T\subset Q$ is an open set and a hypersurface $\Gamma\subset\re^n$ is given by the parametrization $F:T\to \Gamma$ 
$$
  F(\varphi)=r(\varphi)H(\varphi),
$$
where $r$ is some smooth function $r:T\to (0,\infty).$ The surface element in this parametrization calculates as
$$
  \det\left(\left\langle\frac{\partial F}{\partial \varphi_i},\frac{\partial F}{\partial \varphi_j}\right\rangle\right)_{i,j=1,\ldots,n-1} =r^{2n-4}\left(\sum_{i=1}^{n-1}\Big(\frac{\partial r}{\partial \varphi_i}\Big)^2d_1\cdots \widehat{d_i}\cdots d_{n-1}+r^2\prod_{i=1}^{n-1}d_i\right),
$$
where $\widehat{d_i}$ means that $d_i$ should be omitted in the product. In particular, since $d_i>0,$ the surface element is bigger than $r^{n-1}g(\varphi).$
\end{lemma}

\begin{proof}
Using the relations $\langle H;H\rangle =1$ and $\left\langle H;\frac{\partial H}{\partial \varphi_i}\right\rangle=0,$
one obtains that
$$
 M_{ij}:=\left\langle\frac{\partial F}{\partial \varphi_i},\frac{\partial F}{\partial \varphi_j}\right\rangle=\frac{\partial r}{\partial \varphi_i}\frac{\partial r}{\partial\varphi_j}+r^2g_{ij}=\frac{\partial r}{\partial \varphi_i}\frac{\partial r}{\partial\varphi_j}+r^2\delta_{ij}d_i\,.
$$
Thus the matrix $M$ with entries $M_{ij}$ is of the form $M=A+r^2 D$ where $A$ has rank $1$ and $D$ is diagonal with entries $d_i$. Thus using the linearity of the determinant in the columns one obtains that (since no matrix with two columns of $A$ survives when developing the determinant succesively with respect to the columns)
$$
  \det M=\sum_{i=1}^{n-1}\det \widehat{A_i}+\det r^2 D,
$$
where $\widehat{A_i}$ is the matrix obtained from $r^2D$ by replacing the $i$-th column of $r^2D$ by the $i$-th column of $A.$ From this the lemma follows.
\end{proof}

\smallskip

We now use the hypershperical coordinates to deal with domains starshaped with respect to the origin. By definition, a bounded open Lipschitz domain $\Omega\in\re^n$ is starshaped with respect to the origin if there exists a function $R:S^{n-1}\to (0,\infty)$ 
such that
\begin{equation}
 \label{definition:starshaped domain}
  \Omega=\{0\}\cup\left\{x\in\re^n\backslash\{0\}\,:\, \frac{x}{|x|}=\theta\in S^{n-1},\,0<|x|<R(\theta)\right\}
\end{equation}
We shall call $R$ the defining function of $\Omega.$ Note that $R$ is not necessarily Lipschitz, even if $\Omega$ is, and it might even be discontinuous (for example $\Omega=\left(B_{r_2}(0)\cap\{x_2>0\}\right)\cup \left(B_{r_1}(0)\cap\{x_2<0\}\right)\subset\re^2,$ with $r_1<r_2$). But we will alway assume that $R$ is also Lipschitz. It follows from the relation \eqref{eq:jacobian determinant of polar coor} that the volume of a starshaped domain calculates as
\begin{equation}
 \label{eq:volume of starshaped domain}
  \mathcal{L}^n(\Omega)=\frac{1}{n}\int_{S^{n-1}}R^n\dH
\end{equation}
For an almost everywhere differentiable function $f:S^{n-1}\to \re$ we recall that the norm of the covariant gradient of $f$ at $p$ on the manifold $S^{n-1}$ can be calculated as
\begin{equation}
 \label{definition:derivative on S n-1}
  \left|\nabla_{S^{n-1}}f\right|^2=\sum_{i=1}^{n-1}\left|\nabla_{E_i} f\right|^2,
\end{equation}
where $\left\{E_1,\ldots,E_{n-1}\right\}$ is any orthonormal basis of $T_p S^{n-1},$ and $\nabla_E f$ is the derivative in direction $E.$ 

\begin{lemma}
\label{lemma1 for starshaped}
Let $\Omega\subset\re^n$ be a bounded open Lipschitz set, starshaped with respect to the origin and defining function $R$ as in \eqref{definition:starshaped domain}. Assume also that $R$ is Lipschitz. Then for any continuous function $g:(0,\infty)\to \re$ it holds that
\begin{equation}
 \label{eq:lemma 1 for starshaped}
  \intdelomega g(|x|)\dH(x)=\int_{S^{n-1}}(g\circ R)R^{n-1}\sqrt{1+\frac{1}{R^2}\left|\nabla_{S^{n-1}}R\right|^2}\dH.
\end{equation}
\end{lemma}

The assumption that $g$ is continuous can be reduced, but we will need the lemma only for $g(t)=t^p.$
\smallskip

\begin{proof}
We use the hyperspherical coordinates $H,$ definitions of $d_i$ and $g,$ respectively their properties (summarized in the beginning of this section). Let $\varphi\in Q,$ $p=H(\varphi)$ and 
$$
  d_i(\varphi)=\left|\frac{\partial H}{\partial \varphi_i}\right|^2,\quad\text{ then }\quad E_i(p)=\frac{1}{\sqrt{d_i(\varphi)}}\frac{\partial H}{\partial\varphi_i},\quad i=1,\ldots,n-1
$$
form an orthonormal basis of $T_pS^{n-1}.$  Thus we obtain from \eqref{definition:derivative on S n-1} (we can assume that $R$ has been extended to a neighborhood of $S^{n-1}$) that at $p=H(\varphi)$
\begin{equation}
 \begin{split}
  \left|\nabla_{S^{n-1}}R\right|^2=&\sum_{i=1}^{n-1}\frac{1}{d_i(\varphi)}\Big|\Big\langle\nabla R(H(\varphi));\frac{\partial H}{\partial \varphi_i}\Big\rangle\Big|^2=\sum_{i=1}^{n-1}\frac{1}{d_i(\varphi)}\Big[\frac{\partial }{\partial\varphi_i}(R\circ H)\Big]^2
  \smallskip \\
  =&\sum_{i=1}^{n-1}\frac{1}{d_i(\varphi)}\Big[\frac{\partial r }{\partial\varphi_i}\Big]^2
  \end{split}
  \label{eq:derivative on S n-1 at H}
\end{equation}
where we define for $\varphi\in Q$
$$
  r(\varphi):=R(H(\varphi))\quad\text{ and }\quad F(\varphi)=r(\varphi)H(\varphi).
$$
Hence $F:Q\to\delomega$ is a parametrization of $\delomega.$ Using Lemma \ref{lemma:parametrization r H} we obtain that
\begin{align*}
 \sqrt{\det\Big(\Big\langle\frac{\partial F}{\partial \varphi_i},\frac{\partial F}{\partial \varphi_j}\Big\rangle\Big)}
 =&
 r^{n-1}\sqrt{d_1d_2\cdots d_{n-1}}\sqrt{1+\frac{1}{r^2}\sum_{i=1}^{n-1}\frac{1}{d_i}\Big(\frac{\partial r}{\partial \varphi_i}\Big)^2}
 \smallskip \\
 =&r^{n-1}\sqrt{1+\frac{1}{r^2}\sum_{i=1}^{n-1}\frac{1}{d_i}\Big(\frac{\partial r}{\partial \varphi_i}\Big)^2}\, g(\varphi).
\end{align*}
Therefore, using the parametrization $F$ to calculate the left side of \eqref{eq:lemma 1 for starshaped}, respectively the parametrization $H$ to calculate the right side and \eqref{eq:derivative on S n-1 at H}, the lemma follows.
\end{proof}

\begin{lemma}
\label{lemma 2 for starshaped}
Let $R$ be a Lipschitz function, mapping $S^{n-1}$ to $(0,\infty),$ $\alpha\in \re$ and  define $Z:S^{n-1}\to\re$ by $Z=R^{\alpha}.$ Then the following identity holds
$$
  \alpha^2\frac{1}{R^2}\left|\nabla_{S^{n-1}}R\right|^2=\frac{1}{Z^2}\left|\nabla_{S^{n-1}}Z\right|^2.
$$
In particular if $p\in\re$ is such that $p+n-1\neq 0$ and
$$
  Z=R^{\frac{p+n-1}{n-1}},
$$
then
$$
  R^{p+n-1}\sqrt{1+\frac{1}{R^2}\left|\nabla_{S^{n-1}}R\right|^2}
  =Z^{n-1}\sqrt{1+\frac{1}{Z^2}\left|\nabla_{S^{n-1}}Z\right|^2\left(\frac{n-1}{p+n-1}\right)^2}
$$
\end{lemma}

\begin{proof}
Let $E\in T_p S^{n-1}$ be a tangent vector. Then we get
$$
  \nabla_E Z=\langle\nabla Z,E\rangle=\alpha R^{\alpha-1}\langle\nabla R,E\rangle
$$
and it follows from \eqref{definition:derivative on S n-1} that $|\nabla_{S^{n-1}}Z|^2=\alpha^2 R^{2\alpha-2}|\nabla_{S^{n-1}}R|^2.$ From this the lemma follows easily.
\end{proof}


\bigskip

\noindent\textbf{Acknowledgements} The author was supported by the Chilean Fondecyt Iniciaci\'on grant nr. 11150017. He would also like to thank Friedemann Brock for some helpful discussions related to the simplification of the proofs in \cite{Alvino Brock Mercaldo etc.}.


\begin{thebibliography}{200}
\small{

\bibitem{Alvino Brock Mercaldo etc.} A. Alvino, F. Brock, F. Chiacchio, A. Mercaldo and M.R. Posteraro, Some isoperimetric inequalities on $\re^n$ with respect to weights $|x|^{\alpha}$, \textit{J. Math. Anal. Appl.} (2017) http://dx.doi.org/10.1016/j.jmaa.2017.01.085.
\vspace{-.25cm}

\bibitem{Adi-Sandeep} Adimurthi A. and Sandeep K., A singular Moser-Trudinger embedding and its applications, \textit{NoDEA Nonlinear Differential Equations Appl.}, 13 (2007), no. 5-6, 585--603.
\vspace{-.25cm}


\bibitem{Betta and alt.}Betta M.F., Brock F., Mercaldo A. and Posteraro M.R., A weighted isoperimetric inequality and applications to symmetrization, \textit{J. of Inequal. and Appl.}, 4 (1999), 215--240.
\vspace{-.25cm}

\bibitem{Chambers et alt radial p}
Boyer W., Brown B., Chambers G.R., Loving A. and Tammen S., Isoperimetric Regions in $\mathbb{R}^n$ with density $r^p,$  arXiv:1504.01720.
\vspace{-.25cm}

\bibitem{Rosales Canete Bayle Morgan}Bayle V., Ca\~nete A. , Morgan F. and Rosales C., 
On the isoperimetric problem in Euclidean space with density,
\textit{Calc. Var. Partial Differential Equations} 31 (2008), no. 1, 27--46.
\vspace{-.25cm}


\bibitem{Cabre 1} Cabr\'e X. and Ros-Oton X., Sobolev and isoperimetric inequalities with monomial weights, \textit{J. Differential Equations} 255 (2013), no. 11, 4312--4336.
 
\vspace{-.25cm}

\bibitem{Cabre 2}
 Cabr\'e X., Ros-Oton X. and Serra J., Euclidean balls solve some isoperimetric problems with nonradial weights, \textit{C. R. Math. Acad. Sci. Paris} 350 (2012), no. 21-22, 945--947.
\vspace{-.25cm}


\bibitem{Canete-Miranda-Wittone}Ca\~nete A., Miranda M. and Vittone D., 
Some isoperimetric problems in planes with density, \textit{
J. Geom. Anal.}, 20 (2010), no. 2, 243--290. 
\vspace{-.25cm}




\bibitem{Carroll-Jacob-Quinn-Walters}Carroll C., Jacob A. Adam, Quinn C. and Walters R.,
The isoperimetric problem on planes with density, \textit{
Bull. Aust. Math. Soc.}, 78 (2008), no. 2, 177--197.
\vspace{-.25cm}




\bibitem{Chambers Log Convex}Chambers G.R., Proof of the log-convex density conjecture,  	arXiv:1311.4012.
\vspace{-.25cm}

\bibitem{Csato Diff IntegralEq}Csat\'o G., An isoperimetric problem with density and the Hardy-Sobolev inequality in $\re^2$, \textit{Differential Integral Equations}, 28 (2015), no. 9/10, 971--988.
\vspace{-.25cm}

\bibitem{Csato Roy Calc Var}Csat\'o G. and Roy P., Extremal functions for the singular Moser-Trudinger inequality in 2 dimensions, \textit{Calc. Var. Partial Differential Equations}, 54 (2015), no. 2, 2341--2366.
\vspace{-.25cm}
%
\bibitem{Csato Roy Comm PDE}Csat\'o G. and Roy P., The singular Moser-Trudinger inequality on simply connected domains, \textit{Communications in Partial Differential Equations}, to appear.
\vspace{-.25cm}

\bibitem{Dahlberg-Dubbs-Newkirk-Tran}Dahlberg J., Dubbs A., Newkirk E. and Tran H., 
Isoperimetric regions in the plane with density $r^p$, \textit{
New York J. Math.}, 16 (2010), 31--51.
\vspace{-.25cm}

\bibitem{Diaz-Harman-Howe-Thompson} D\'iaz A., Harman N., Howe S. and  Thompson D., Isoperimetric problems in sectors with
density, \textit{Adv. Geom.}, 12 (2012), 589--619.
\vspace{-.25cm}

\bibitem{do Carmo Barbosa} Barbosa J.L. and do Carmo M., Stability of hypersurfaces with constant mean curvature, \textit{Math. Z.}, 185 (1984), no. 3, 339--353.
\vspace{-.25cm}


\bibitem{Figalli-Maggi.LogConvex}Figalli A. and Maggi F., 
On the isoperimetric problem for radial log-convex densities, 
\textit{Calc. Var. Partial Differential Equations}, 48 (2013), no. 3-4, 447--489.
\vspace{-.25cm}



\bibitem{Flucher}Flucher M., Extremal functions for the Trudinger-Moser inequality in 2 dimensions, \textit{Comment. Math. Helvetici}, 67 (1992), 471--497.
\vspace{-.25cm}


\bibitem{Fusco-Maggi-Pratelli}Fusco N., Maggi F. and Pratelli A.,
On the isoperimetric problem with respect to a mixed Euclidean-Gaussian
density, 
\textit{J. Funct. Anal.}, 260 (2011), no. 12, 3678--3717.
\vspace{-.25cm}

\bibitem{Giosia ...morgans group rk and rm} Di Giosia L., Habib J., Kenigsberg L., Pittman D and Zhu W, \textit{Balls Isoperimetric in $\re^n$ with Volume and Perimeter Densities $r^m$ and $r^k$}, arXiv:1610.05830v1. 
\vspace{-.25cm}











\bibitem{Morgan-Regularity}Morgan F., Regularity of isoperimetric hypersurfaces in Riemannian manifolds, \textit{Trans. Amer. Math. Soc.}, 355 (2003), no. 12, 5041--5052.
\vspace{-.25cm}

\bibitem{Morgan Vari Formulae}Morgan F., http://sites.williams.edu/Morgan/2010/06/22/variation-formulae-for-perimeter-and-volume-densities/.
\vspace{-.25cm}

\bibitem{Morgan-Pratelli}Morgan F. and Pratelli A., Existence of isoperimetric regions in $\mathbb{R}^n$ with density, \textit{ Ann. Global Anal. Geom.}, 43 (2013), no. 4, 331--365. 
\vspace{-.25cm}

\bibitem{Walter} Walter W., \textit{Ordinary differential equations}, English translation, Springer, 1998.








}
\end{thebibliography}
\end{document}